\documentclass[letterpaper, 10 pt, conference]{ieeeconf}
\IEEEoverridecommandlockouts 
\usepackage{amsmath, amssymb}
\usepackage{graphicx} 
\usepackage{cleveref}
\usepackage{empheq, xcolor}

\newcommand{\E}{\mathbb{E}}

\newcommand{\pair}[1]{\ensuremath{\left\langle #1 \right\rangle}}
\renewcommand{\Re}{\ensuremath{\mathbb{R}}}

\newcommand{\deriv}[2]{\ensuremath{\frac{\partial #1}{\partial #2}}}

\renewcommand{\div}{\mathrm{div}}

\newcommand{\Sph}{\mathsf{S}}

\newcommand{\g}{\ensuremath{\mathsf{g}}}

\newcommand{\EditTL}[1]{{\color{blue}\protect #1}}
\renewcommand{\EditTL}[1]{}

\newtheorem{theorem}{Theorem}
\newtheorem{lemma}{Lemma}

\newtheorem{remark}{Remark}

\newtheorem{example}{Example}

\title{\LARGE\bf Global Tensor Field Formulation of the Fokker–Planck Equation on Riemannian Manifolds}

\author{Taeyoung Lee and Gregory S. Chirikjian
    \thanks{Taeyoung Lee, Mechanical and Aerospace Engineering, George Washington University, Washington, DC 20052, {\tt tylee@gwu.edu}}
    \thanks{Gregory S. Chirikjian, Willis F. Harrington Professor of Mechanical Engineering, University of Delaware, Newark DE 19716 {\tt gchirik@udel.edu}}
    \thanks{\textsuperscript{\footnotesize\ensuremath{*}}This research has been supported in part by AFOSR MURI FA9550-23-1-0400, and ONR N00014-23-1-2850.}
}

\begin{document}

\maketitle

\begin{abstract}
    This paper presents a global, coordinate-free formulation of the Fokker–Planck equation on Riemannian manifolds.  
    In the Stratonovich formulation, the infinitesimal generator is expressed intrinsically through Lie derivatives, and its adjoint is derived via the divergence theorem, yielding a concise geometric form of the Fokker–Planck equation.  
    In the It\^{o} formulation, a \emph{diffusion tensor field} is introduced to generalize the Euclidean diffusion matrix, and a tensor-field-based analysis establishes an intrinsic double-divergence representation of the Fokker–Planck equation.  
    The proposed framework provides a globally valid and geometrically consistent interpretation of diffusion and probability transport on Riemannian manifolds, supported by compact and intuitive proofs.
\end{abstract}

\section{Introduction}

The Fokker-Planck equation plays a central role in connecting stochastic dynamics with probabilistic and analytical descriptions of diffusion~\cite{bogachev2022fokker,chirikjian2009stochastic}. 
While a stochastic differential equation governs the evolution of individual sample paths, the corresponding Fokker-Planck equation provides a deterministic partial differential equation that describes how the probability density of the process evolves over time. 
It thus serves as a fundamental bridge between microscopic random motion and macroscopic statistical behavior, and is a critical tool in estimation, filtering, and machine learning. 
On Riemannian manifolds, the Fokker–Planck equation captures how curvature and geometric structure influence the transport and dispersion of probability, providing a natural extension of conservation laws and diffusion phenomena to curved spaces. 

The seminal works of Yosida~\cite{yosida1949integration} and It\^o~\cite{ito1950stochastic,ito1953stochastic} established the mathematical foundation for formulating the Fokker-Planck equation on differentiable manifolds.  
Yosida provided an analytic framework by extending diffusion operators and semigroup theory to compact Riemannian spaces, while It\^o introduced the stochastic differential formulation that defined diffusion processes intrinsically and derived their corresponding Fokker-Planck equations from the generator.  
Together, these works unified probabilistic, geometric, and analytic perspectives, laying the groundwork for modern stochastic analysis on manifolds~\cite{chirikjian2009stochastic,elworthy1982stochastic,emery1989stochastic,hsu2002stochastic}.

Despite these foundational developments, classical formulations of the Fokker-Planck equation on manifolds are largely expressed in local coordinates, where geometric quantities such as the metric, connection, and Christoffel symbols appear explicitly.  
While this coordinate-dependent approach offers analytic tractability, it obscures the underlying geometric structure and complicates extensions to globally defined processes on manifolds with nontrivial topology.  
A fully intrinsic, coordinate-free formulation is therefore desirable to capture the global behavior of diffusion and transport in a manner consistent with the manifold’s geometry, independent of any particular parameterization.  

This paper develops a global, coordinate-free formulation of the Fokker-Planck equation on a Riemannian manifold directly from first principles of stochastic dynamics and differential geometry.  
In the Stratonovich formulation, the infinitesimal generator is expressed intrinsically in terms of Lie derivatives along vector fields, and its adjoint is obtained using the divergence theorem on manifolds, leading to a compact and geometrically consistent expression of the Fokker-Planck equation (\Cref{thm:FP_Strat}).

In the It\^{o} formulation, we introduce the concept of a \emph{diffusion tensor field}, which generalizes the conventional diffusion matrix in Euclidean space through tensor products of diffusion vector fields, providing an intrinsic representation of anisotropic diffusion on curved spaces (\Cref{thm:generator_Ito}).  
To construct the It\^{o} Fokker–Planck equation in this intrinsic form, we present a tensor field framework, for example, by utilizing the extensions of divergence and covariant derivatives to higher-order tensor fields and an integration-by-parts identity for tensor fields (\Cref{lem:IBP_tensor}).
This enables the It\^{o} Fokker–Planck equation to be represented compactly as a double divergence, which encapsulates curvature-adjusted probability spreading governed by the manifold’s metric structure (\Cref{thm:FP_Ito}).  
More importantly, the proposed diffusion tensor fields geometric and physical interpretation of the Fokker-Planck equation on a manifold (\Cref{rem:D,rem:cont}).
For example, we show that an isotropic diffusion leads to the Brownian motion with an appropriate choice of the drift field (\Cref{ex:Brownian_Strat,ex:Brownian_Ito}).

Recently, Huang~\cite{huang2022riemannian} presented a coordinate-free form of the Fokker–Planck equation in the Stratonovich sense by converting the coordinate-based expression of~\cite{chirikjian2009stochastic} into intrinsic operators.  
In contrast, the Stratonovich formulation proposed here is derived directly from geometric first principles. 
Next, the proposed tensor field formulation of the It\^o Fokker-Planck equation has not been reported. 

\EditTL{In summary, the main contributions of this paper lie in the tensor-field formulation of the Fokker–Planck equation, which offers an elegant and compact expression supported by concise geometric proofs grounded in intrinsic differential operators, and enables Bayesian estimation on Riemannian manifolds.}

\section{Background} \label{sec:Background}

We briefly review key concepts from differential geometry and tensor analysis~\cite{lee2006riemannian,marsden2013introduction,petersen2006riemannian}.

\subsection{Riemannian Manifold}\label{sec:diff_geo}

Let $(M,\g)$ be an $n$-dimensional Riemannian manifold, where $M$ is a smooth manifold and $\g$ is a Riemannian metric. 
For a point $p \in M$, the \emph{tangent space} $T_p M$ is the vector space consisting of all possible tangent vectors at $p$. 
The \emph{tangent bundle} $TM$ is the disjoint union of the tangent spaces at all points of the manifold.
A \emph{vector field} is a smooth map $X:M\to TM$ such that $X(p)\in T_pM$ for all $p\in M$.
The set of all smooth vector fields on $M$ is denoted by $\mathfrak{X}(M)$.

The \emph{cotangent space} $T_p^*M$ is the dual space of $T_pM$, consisting of all linear maps $\alpha:T_pM \to \Re$.
An element of $T_p^*M$ is called a \emph{covector} or \emph{one-form}.
A \emph{covector field} is a smooth section $\alpha:M\to T^*M$, assigning to each $p\in M$ a covector $\alpha(p)\in T_p^*M$.
The collection of all covector fields on $M$ is denoted by $\Omega^1(M)$.

We now review several fundamental differential operators. 
Let $C^\infty(M)$ denote the set of smooth real-valued functions on $M$.
For $f\in C^\infty(M)$ and $X\in\mathfrak{X}(M)$, the \emph{Lie derivative of $f$ along $X$} is defined by
\begin{align}
    \mathcal{L}_X f = X[f] \triangleq df(X) \in C^\infty(M), \label{eqn:lie_derivative}
\end{align}
where $df$ is the differential of $f$, viewed as a covector field.

Given a Riemannian metric, there exists a unique torsion-free connection 
$\nabla:\mathfrak{X}(M)\times\mathfrak{X}(M)\rightarrow\mathfrak{X}(M)$ that is compatible with the metric, called the \emph{Levi--Civita connection}.
For $X,Y\in\mathfrak{X}(M)$, $\nabla_X Y$ is referred to as the \textit{covariant derivative}.

The \emph{divergence} $\div_\mu:\mathfrak{X}(M)\rightarrow C^\infty(M)$ of a vector field $X$ is the trace of the linear map $Y\mapsto \nabla_Y X$~\cite[page~88]{lee2006riemannian}, given by
\begin{align}
    \div_\mu X = \mathrm{tr}\big(Y \mapsto \nabla_Y X\big). \label{eqn:div}
\end{align}
The \emph{Hessian} is the symmetric $(0,2)$-tensor defined by
\begin{align}
    \mathrm{Hess}_f(X,Y) & 
    = X[Y[f]] - (\nabla_X Y)[f] \nonumber \\
    & = \mathrm{Hess}_f(Y,X).\label{eqn:Hess}
\end{align}

\subsection{Tensor Fields} \label{sec:appendix_tensor}

Given a Riemannian manifold $(M,\g)$, a \emph{tensor field of type $(r,s)$} on $M$, denoted by $\mathcal{T}^{(r,s)}(M)$, is a smooth, $C^\infty(M)$-multilinear map:
\begin{align}
    T :&
    \underbrace{\Omega^1(M) \times \cdots \times \Omega^1(M)}_{r~\text{times}}
    \times
    \underbrace{\mathfrak{X}(M) \times \cdots \times \mathfrak{X}(M)}_{s~\text{times}}\nonumber\\
       & \to C^\infty(M),
\end{align}
where the type $(r,s)$ specifies the number of \textit{contravariant} (upper) and \textit{covariant} (lower) indices, respectively.
For example, a scalar field is a $(0,0)$ tensor, a vector field is a $(1,0)$ tensor, and a covector field (or one-form) is a $(0,1)$ tensor.
In local coordinates $\{x^i\}$, a $(r,s)$ tensor $T$ has components $T^{i_1\cdots i_r}{}_{j_1\cdots j_s}$.

Given two tensor fields $T_1 \in \mathcal{T}^{(r_1,s_1)}(M)$ and $T_2 \in \mathcal{T}^{(r_2,s_2)}(M)$,
their \emph{tensor product}, denoted by $T_1 \otimes T_2$, is a tensor field of type $(r_1+r_2,\, s_1+s_2)$ defined by
\begin{align}
    &(T_1 \otimes T_2)
    (\alpha^1,\ldots,\alpha^{r_1+r_2}, v_1,\ldots,v_{s_1+s_2}) \nonumber\\
    &= T_1(\alpha^1,\ldots,\alpha^{r_1}, v_1,\ldots,v_{s_1})\nonumber\\
    & \quad\times   T_2(\alpha^{r_1+1},\ldots,\alpha^{r_1+r_2}, v_{s_1+1},\ldots,v_{s_1+s_2}),
    \label{eqn:tensor_product_eval}
\end{align}
for $\alpha^i\in\Omega^1(M)$ and $v_i\in\mathfrak{X}(M)$.
The tensor product acts on the concatenation of the arguments of $T_1$ and $T_2$, evaluating each separately and multiplying the results.
For instance, if $X$ and $Y$ are vector fields (type $(1,0)$), then
$X \otimes Y$ is a $(2,0)$ tensor field defined by
\[
    (X \otimes Y)(\alpha,\beta) = \alpha(X)\,\beta(Y),
\]
for any $\alpha,\beta\in\Omega^1(M)$.

A \emph{contraction} of a tensor field is the operation of pairing one contravariant index with one covariant index and summing over it.
More generally, given two tensor fields 
$T_1 \in \mathcal{T}^{(r_1,s_1)}(M)$ and 
$T_2 \in \mathcal{T}^{(r_2,s_2)}(M)$,
one can form their tensor product $T_1\otimes T_2$ of type $(r_1+r_2,\, s_1+s_2)$ and then contract any matching pair of upper and lower indices to obtain a tensor of reduced type.
This operation is denoted generically by $T_1:T_2$, when the contracted indices are clear from context.  
For example, if $T_1\in\mathcal{T}^{(2,0)}(M)$ and $T_2\in\mathcal{T}^{(0,2)}(M)$, 
\begin{align}
    T_1: T_2 = (T_1)^{ij} (T_2)_{ij}. \label{eqn:contraction_2}
\end{align}
Throughout this paper, we adopt the Einstein convention: we take the sum of any index that appears in both of the upper slot and the lower slot.

The covariant derivative of a vector field extends naturally to tensor fields.
Specifically, for a tensor field $T$ of type $(r,s)$, the \textit{covariant derivative} $\nabla T$ is a tensor of type $(r,s+1)$ defined by
\begin{align}
    &(\nabla T)(\alpha^1,\ldots,\alpha^r, v_1,\ldots,v_s, X) \nonumber\\
    &\triangleq (\nabla_X T)(\alpha^1,\ldots,\alpha^r, v_1,\ldots,v_s) \nonumber\\
    &= X\!\left[T(\alpha^1,\ldots,\alpha^r, v_1,\ldots,v_s)\right] \nonumber\\
    &\quad - \sum_{j=1}^r T(\alpha^1,\ldots,\nabla_X \alpha^j,\ldots,\alpha^r, v_1,\ldots,v_s) \nonumber\\
    &\quad - \sum_{i=1}^s T(\alpha^1,\ldots,\alpha^r, v_1,\ldots,\nabla_X v_i,\ldots,v_s),
    \label{eqn:covariant_derivative_tensor}
\end{align}
for vector fields $X,v_1,\ldots,v_s\in \mathfrak{X}(M)$ and covector fields $\alpha^1,\ldots,\alpha^r\in\Omega^1(M)$.
As such, the covariant derivative of a tensor field increases the number of covariant slots by one.
For example, let $f\in C^\infty(M)$ be a $(0,0)$ tensor. 
Using \eqref{eqn:covariant_derivative_tensor}, $\nabla f$ is a $(0,1)$ tensor evaluated by
\begin{align}
    (\nabla f)(X) = X[f], \label{eqn:nabla_f_X}
\end{align}
for any $X\in\mathfrak{X}(M)$.
And applying \eqref{eqn:covariant_derivative_tensor} again, $\nabla (\nabla f)$ is a $(0,2)$ tensor obtained by
\begin{align*}
    (\nabla (\nabla f))(X,Y) & = Y[(\nabla f)(X)] - (\nabla f)(\nabla_Y X),\\
                             & = Y[X[f]] - (\nabla_Y X)[f],
\end{align*}
for any $X,Y\in\mathfrak{X}(M)$, where the second equality is due to \eqref{eqn:nabla_f_X}.
Comparing this with \eqref{eqn:Hess}, and noting that the Hessian is symmetric, we obtain
\begin{align}
    \mathrm{Hess}_f(X,Y) =  (\nabla (\nabla f))(X,Y), \label{eqn:Hess_cov}
\end{align}
which shows that the Hessian corresponds to the double covariant derivative.
But, note that the inner covariant derivative of \eqref{eqn:Hess_cov} acts on a $(0,0)$ tensor field, and the outer covariant derivative acts on a $(0,1)$ tensor field. 

For tensor fields $T,S$ of arbitrary type, the following product rule (Leibniz rule) holds:
\begin{align}
    \nabla_X (T\otimes S) = (\nabla_X T)\otimes S + T\otimes (\nabla_X S). \label{eqn:cov_tensor_product}
\end{align}

The \emph{divergence of a tensor field} $T$ is obtained by contracting the last contravariant index of $\nabla T$ with its newly introduced covariant index.
If $T$ is a $(r,s)$ tensor, then $\nabla T$ is a $(r, s+1)$ tensor, and this contraction yields a $(r-1, s)$ tensor, denoted $\div_\mu T$.
More explicitly, the divergence $\div_\mu T$ is defined as the trace of the $(1,1)$ tensor map formed by fixing all arguments \textit{except} the last contravariant and last (new) covariant slots:
\begin{align}
    &(\div_\mu T) (\alpha^1,\ldots, \alpha^{r-1}, v_1, \ldots, v_{s}) \nonumber \\
    & \quad = \mathrm{tr}\left( (\alpha, X) \mapsto (\nabla T)(\alpha^1, \ldots, \alpha^{r-1}, \alpha, v_1, \ldots, v_s, X) \right) \nonumber \\
    & \quad = \sum_{i=1}^n (\nabla_{E_i} T)(\alpha^1, \ldots, \alpha^{r-1}, E^i, v_1, \ldots, v_s),\label{eqn:div_tensor}
\end{align}
where $\{E_i\}$ is a local orthonormal frame and $\{E^i\}$ is its dual coframe. 
Thus, the divergence of a tensor field decreases the contravariant slot of the tensor by 1.

For example, if $X\in\mathfrak{X}(M)$ is a vector field, or $(1,0)$ tensor field ($r=1, s=0$), this definition reduces to
\[
    \div_\mu (X) = \sum_{i=1}^n (\nabla_{E_i} X)(E^i) = \sum_{i=1}^n E^i(\nabla_{E_i} X),
\]
which is consistent with \eqref{eqn:div}.

Let $T_1\in\mathcal{T}^{(r,s)}(M)$ and $T_2\in\mathcal{T}^{(s,r)}(M)$.
As they have the opposite tensor type, their contraction reduces to a scalar.
The $L^2$-inner product between tensor fields of opposite type, denoted by $\pair{\cdot,\cdot}_{\mathcal{T}}$, is defined as the integral of their contraction over $M$:
\begin{align}
    \pair{T_1, T_2}_{\mathcal{T}} = \int_M (T_1: T_2) \,d\mu(x), \label{eqn:inner_L2_tensor}
\end{align}
where $\mu$ denotes the volume form of $M$.

\section{Intrinsic Fokker-Planck Equations}\label{sec:BMM}

\subsection{Stochastic Processes and the Fokker-Planck Equation}

Let $x(t)$ be a continuous stochastic process on a Riemannian manifold $(M,\g)$.
The dynamics of this process are characterized by its \emph{infinitesimal generator} $\mathcal{A}:C^\infty(M)\to C^\infty(M)$, 
a differential operator that gives the expected instantaneous rate of change of any smooth function along the process.
More explicitly, for any $f\in C^\infty(M)$ and $x_0 \in M$,
\begin{align}
    (\mathcal{A} f)(x_0) = \lim_{t \to 0^+} \frac{\E[f(x(t))\,|\, x(0) = x_0] - f(x_0)}{t},
\end{align}
which encodes the local drift and diffusion properties of the process $x(t)$.

For any $p, q \in C^\infty_c(M)$, that is, smooth functions on $M$ with compact support, let the $L^2$-inner product be defined as
\begin{align}
    \pair{p, q} = \int_M p(x) q(x) d\mu(x), \label{eqn:inner_L2}
\end{align}
where $d\mu$ is the Riemannian volume form on $(M,\g)$.

The \textit{adjoint} of the generator, denoted $\mathcal{A}^*$, is defined with respect to the $L^2$-inner product on $M$.
Specifically, the dual $\mathcal{A}^*: C^\infty_c(M)\rightarrow C_c^\infty(M)$ is the unique operator satisfying
\begin{align}
    \langle p, \mathcal{A}q \rangle = \langle \mathcal{A}^* p, q \rangle, \label{eqn:dual}
\end{align}
for any $p,q\in C^\infty_c(M)$.

This adjoint operator is crucial for describing the evolution of the probability density function, denoted by $p_t(x) \triangleq p(t,x)$ of the process.
The expected value of a test function $q(x(t))$ evolves in time as
\begin{align*}
    \frac{d}{dt} \mathbb{E}[q(x(t))] = \frac{d}{dt} \int_M q(x) p_t(x) \, d\mu(x) = \left\langle \frac{\partial p_t}{\partial t}, q \right\rangle.
\end{align*}
By the definition of the generator, this time derivative is also equal to
\begin{align*}
    \frac{d}{dt} \mathbb{E}[q(x(t))] = \mathbb{E}[\mathcal{A}q(x(t))] = \langle \mathcal{A}q, p_t \rangle
    = \langle \mathcal{A}^* p_t, q \rangle,
\end{align*}
where the last equality is from the definition of the adjoint operator, given by \eqref{eqn:dual}.
Equating the above two equations,
\begin{align}
    \left\langle \frac{\partial p_t}{\partial t}, q \right\rangle = \langle \mathcal{A}^* p_t, q \rangle.
\end{align}
Since this equality must hold for any arbitrary function $q \in C^\infty_c(M)$, we obtain the \textit{Fokker-Planck equation} in an abstract form as
\begin{align}
    \frac{\partial p_t}{\partial t} = \mathcal{A}^* p_t. \label{eqn:FP_A}
\end{align}
This provides a deterministic partial differential equation describing the time evolution of the probability density $p(t,x)$ on the manifold $M$.

Next, we derive an explicit expression of the Fokker-Planck equation for a stochastic process defined by a stochastic differential equation (SDE). 
We first recover the Stratonovich formulation and then present the Itô formulation in terms of tensor fields.

\subsection{Fokker-Planck Equation for Stratonovich SDE}

Consider a stochastic differential equation on a manifold $M$ of the form
\begin{align}
    dx = X(x)\,dt + \sum_{i=1}^m \sigma_i(x) \circ dW_i, \label{eqn:SDE}
\end{align}
where $X, \sigma_1, \ldots, \sigma_m \in \mathfrak{X}(M)$ are smooth vector fields, and $W_i$ are independent real-valued Wiener processes.
This equation is written in the \emph{Stratonovich} sense, which ensures that the solution flow respects the manifold structure and that the standard rules of calculus apply.

\begin{theorem}[Generator of Stratonovich SDE]\label{thm:generator_Strat}
    The infinitesimal generator $\mathcal{A}:C^\infty(M)\rightarrow C^\infty(M)$ of the Stratonovich stochastic differential equation \eqref{eqn:SDE} on $f\in C^\infty(M)$ is given by
    \begin{align}
        \boxed{\mathcal{A} f = X[f] + \frac{1}{2}\sum_{i=1}^m \sigma_i [\sigma_i[f]].} \label{eqn:A_Strat}
    \end{align}
\end{theorem}
\vspace*{0.3cm}
\begin{proof}
    See~\cite[pp. 253]{ikeda1989stochastic} and~\cite[Theorem 2]{lee2025brownian}.
\end{proof}

Next, the following identity is repeatedly used in the subsequent development of the Fokker-Planck equation.
\begin{lemma}\label{lem:IBP}
    Suppose $p,q\in C^\infty_c(M)$ are smooth, compactly supported functions on $M$, and $X\in\mathfrak{X}(M)$ is a vector field. Then,
    \begin{align}
        \pair{p, X[q]} & = \pair{-\div_\mu(pX), q},\label{eqn:IBP}
    \end{align}
    where $\pair{\cdot, \cdot}$ denotes the $L^2$-inner product given by \eqref{eqn:inner_L2}.
\end{lemma}
\begin{proof}
    See \Cref{sec:IBP}.
\end{proof}
Equation~\eqref{eqn:IBP} shows that the adjoint of the Lie derivative with respect to the $L^2$-inner product is given by the divergence operator.

\begin{theorem}[Fokker-Planck Equation for Stratonovich SDE]\label{thm:FP_Strat}
    Consider the stochastic differential equation \eqref{eqn:SDE} defined in the Stratonovich sense. 
    The corresponding Fokker-Planck equation is given by
    \begin{align}
        \boxed{
        \deriv{p_t}{t} = -\div_\mu(p_tX) + \frac{1}{2}\sum_{i=1}^m \div_\mu( \div_\mu(p_t\sigma_i)\, \sigma_i ).}\label{eqn:FP_Strat}
    \end{align}
\end{theorem}
\vspace*{0.3cm}
\begin{proof}
    From \eqref{eqn:dual}, the adjoint of the generator \eqref{eqn:A_Strat} is obtained by
    \begin{align}
        \pair{ \mathcal{A}^*p, q} = \langle p, X[q] + \frac{1}{2}\sum_{i=1}^m \sigma_i[\sigma_i[q]] \rangle, \label{eqn:tmp0}
    \end{align}
    for any $p, q\in C^\infty_c (M)$.
    By repeatedly applying \eqref{eqn:IBP},
    \begin{gather*}
        \pair{p, X[q]} = \pair{ -\div_\mu(pX), q}, \\
        \pair{p, \sigma_i[q]} = \pair{ -\div_\mu(p\sigma_i), q},\\
        \pair{p, \sigma_i[\sigma_i[q]]} = \pair{ \div_\mu( \div_\mu(p\sigma_i)\, \sigma_i ), q}.
    \end{gather*}
    Substituting these into \eqref{eqn:tmp0}, the adjoint operator $\mathcal{A}^*$ is obtained by the right-hand side of \eqref{eqn:FP_Strat}.
    From \eqref{eqn:FP_A}, this shows \eqref{eqn:FP_Strat}.
\end{proof}
The Fokker-Planck equation for a Stratonovich SDE on a manifold has been expressed in local coordinates in, for example, \cite{chirikjian2009stochastic}.
\Cref{thm:FP_Strat} shows that an intrinsic form of the Stratonovich Fokker-Planck equation can be derived elegantly using the integration-by-parts formula \eqref{eqn:IBP}.

\begin{example}[Stratonovich SDE for Brownian Motion]\label{ex:Brownian_Strat}
    Suppose $X = -\frac{1}{2}\sum_{i=1}^n \nabla_{E_i} E_i$ and $\sigma_i = E_i$ for $i \in \{1, \ldots, n\}$.
    Using the algebraic properties of the divergence and the Lie derivative, it follows that both the generator $\mathcal{A}$ in~\eqref{eqn:A_Strat} and its adjoint in~\eqref{eqn:FP_Strat} reduce to one half of the Laplace--Beltrami operator.
    This corresponds to the Brownian motion on the manifold, thereby recovering the result of~\cite[Theorem~4,~(39)]{lee2025brownian}.
\end{example}



\subsection{Fokker-Planck Equation for It\^{o} SDE}\label{sec:FP_Ito}

In this section, we consider a stochastic differential equation on $M$ given by
\begin{align}
    dx = \tilde X(x)\,dt + \sum_{i=1}^m \sigma_i(x)  dW_i, \label{eqn:SDE_Ito}
\end{align}
which is defined according to the It\^{o} sense. 
As in \eqref{eqn:SDE}, $\tilde X, \sigma_1, \ldots, \sigma_m \in \mathfrak{X}(M)$ are smooth vector fields and $W_i$ are independent real-valued Wiener processes. 

We first identify the correspondence between the Stratonovich SDE \eqref{eqn:SDE} and \eqref{eqn:SDE_Ito} as follows.
\begin{theorem}[It\^o--Stratonovich Conversion]\label{thm:Strat_Ito}
    The Stratonovich stochastic differential equation \eqref{eqn:SDE} is equivalent to the It\^o stochastic differential equation \eqref{eqn:SDE_Ito} if the drift vector fields are related by
    \begin{align}
        \tilde X(x) = X(x) + \frac{1}{2} \sum_{i=1}^m \nabla_{\sigma_i(x)} \sigma_i(x). \label{eqn:X_tilde}
    \end{align}
\end{theorem}
\begin{proof}
    See \cite[Theorem 1]{lee2025brownian}.
\end{proof}
The infinitesimal generator of the It\^o SDE \eqref{eqn:SDE_Ito} can be obtained by substituting the conversion formula \eqref{eqn:X_tilde} into the generator of the Stratonovich SDE \eqref{eqn:A_Strat}~\cite{lee2025brownian}.
Alternatively, we present a formulation based on the \emph{diffusion tensor field}, which yields a more interpretable and compact intrinsic expression of the It\^o Fokker-Planck equation on a manifold.

\begin{theorem}[Generator of It\^{o} SDE]\label{thm:generator_Ito}
    The infinitesimal generator $\mathcal{A}:C^\infty(M)\rightarrow C^\infty(M)$ of the It\^{o} stochastic differential equation \eqref{eqn:SDE_Ito} on $f\in C^\infty(M)$ is given by
    \begin{align}
        \boxed{
    \mathcal{A} f = \tilde X[f] + \frac{1}{2} ( D : \mathrm{Hess}_f), \label{eqn:A_Ito_D}}
    \end{align}
    where $D$ is a $(2,0)$ symmetric tensor field on $M$, referred to as \textit{diffusion tensor field}, defined by
    \begin{align}
        D = \sum_{i=1}^m \sigma_i \otimes \sigma_i. \label{eqn:D}
    \end{align}
    In \eqref{eqn:A_Ito_D}, $\mathrm{Hess}_f:\mathfrak{X}(M)\times\mathfrak{X}(M)\rightarrow\Re$ is considered as a $(0,2)$ tensor field, 
    and $:$ denotes the contraction of a pair of tensor fields introduced at \eqref{eqn:contraction_2}.
\end{theorem}
\begin{proof}
    See \Cref{sec:generator_Ito}.
\end{proof}




\begin{remark}\label{rem:D}
    The diffusion tensor field is a smooth, symmetric, positive semi-definite $(0,2)$-tensor that intrinsically characterizes how stochastic perturbations spread over a manifold. 
    Unlike the conventional diffusion matrix defined in Euclidean coordinates, this tensorial formulation generalizes diffusion to curved spaces by expressing it as a metric-compatible tensor field on the tangent bundle.
    It provides a unified geometric interpretation of diffusion as the local covariance of noise vector fields, consistent with the manifold’s curvature and metric structure. 
    Physically, it captures the intrinsic anisotropy and rate of probability propagation constrained by the geometry. 
    In the isotropic case, the diffusion part of the generator reduces to the Laplace-Beltrami operator, corresponding to Brownian motion (\Cref{ex:Brownian_Ito}).
\end{remark}

Next, the following lemma is repeatedly used in the subsequent development of the Fokker-Planck equation in the It\^{o} form. 
\begin{lemma}\label{lem:IBP_tensor}
    Let $T\in\mathcal{T}^{(r,s)(M)}$ and $S\in\mathcal{T}^{(s+1, r)(M)}$ be tensor fields with compact support. 
    Then, 
    \begin{align}
        \langle \nabla T, S \rangle_{\mathcal{T}} = \langle T, - \div_\mu (S) \rangle_{\mathcal{T}},\label{eqn:IBP_tensor}
    \end{align}
    where $\pair{\cdot, \cdot}_{\mathcal{T}}$ is the $L^2$-inner product of tensor fields given by \eqref{eqn:inner_L2_tensor}. 
\end{lemma}
\begin{proof}
    See \Cref{sec:IBP_tensor}.
\end{proof}
This lemma generalizes \Cref{lem:IBP}, showing that the adjoint of the covariant derivative of a tensor field is the negative divergence. 
Equivalently, it extends the classical integration-by-parts formula to arbitrary tensor fields on $M$.

\begin{theorem}[Fokker-Planck Equation for It\^o SDE]\label{thm:FP_Ito}
    Consider the stochastic differential equation \eqref{eqn:SDE_Ito} defined in the It\^o sense. 
    The corresponding Fokker-Planck equation is
    \begin{align}
        \boxed{
        \frac{\partial p_t}{\partial t} 
        = -\div_\mu(p_t\tilde X) + \frac{1}{2}\,\div_\mu(\div_\mu(p_t D)),}
        \label{eqn:FP_Ito}
    \end{align}
    where $D$ is the diffusion tensor field defined in \eqref{eqn:D}.
    The inner divergence in \eqref{eqn:FP_Ito} is taken as in \eqref{eqn:div_tensor} and acts on the $(2,0)$ tensor field $p_t D$, 
    while the outer divergence follows \eqref{eqn:div} and acts on the resulting $(1,0)$ tensor field. 
\end{theorem}

\begin{proof}
    From \eqref{eqn:dual}, the dual of the generator \eqref{eqn:A_Strat} is obtained by
    \begin{align}
        \pair{ \mathcal{A}^*p, q} = \langle p, \tilde X[q] + \frac{1}{2} ( D: \mathrm{Hess}_q ) \rangle, \label{eqn:tmp1}
    \end{align}
    for any $p, q\in C^\infty_c (M)$.

    Using \eqref{eqn:IBP}, the adjoint of the drift is given by
    \begin{align}
        \langle p, \tilde X[q] \rangle = \langle -\div_\mu(p\tilde X), q \rangle. \label{eqn:tmp2}
    \end{align}
    Next, the diffusion part is rearranged into
    \begin{align*}
        \langle p, D:\mathrm{Hess}_q \rangle = \langle pD, \mathrm{Hess}_q \rangle_{\mathcal{T}}
         = \langle pD, \nabla(\nabla q) \rangle_{\mathcal{T}},
    \end{align*}
    where the first equality is from the formulation of the $L^2$-inner product of tensor fields given by \eqref{eqn:inner_L2_tensor}, and the last equality is due to \eqref{eqn:Hess_cov}.
    Applying \eqref{eqn:IBP_tensor}, this is rewritten as
    \begin{align}
        \langle p, D:\mathrm{Hess}_q \rangle & = \langle \div_\mu(\div_\mu(pD)), q \rangle_{\mathcal{T}} \nonumber\\
                                             & = \langle \div_\mu(\div_\mu(pD)), q \rangle,\label{eqn:tmp3}
    \end{align}
    where the $L^2$-inner product for tensor fields of the type $(2,0)$ is reduced to the usual $L^2$-inner product for the scalar fields.

    Substituting \eqref{eqn:tmp2} and \eqref{eqn:tmp3} into \eqref{eqn:tmp1} yields
\begin{align*}
    \pair{ \mathcal{A}^*p, q} &= \left\langle -\div_\mu(p\tilde X) + \frac{1}{2} \div_\mu(\div_\mu(pD)), q \right\rangle.
\end{align*}
Since this must hold for all test functions $q \in C^\infty_c(M)$, we identify the adjoint operation $\mathcal{A}^*$ on $p_t$ as \eqref{eqn:FP_Ito}.
\end{proof}
\begin{example}[It\^o SDE for Brownian Motion]\label{ex:Brownian_Ito}
    Suppose $\tilde X = 0$, and $\sigma_i = E_i$ for $i \in \{1, \ldots, n\}$ such that $D=\sum_{i=1}^n E_i\otimes E_i$.
    Using the algebraic properties of the divergence and the Lie derivative, it follows that both the generator $\mathcal{A}$ in~\eqref{eqn:A_Ito_D} and its adjoint in~\eqref{eqn:FP_Ito} reduce to one half of the Laplace--Beltrami operator.
    This corresponds to the Brownian motion on the manifold, thereby recovering the result of~\cite[Theorem~4,~(40)]{lee2025brownian}.
\end{example}

The Fokker-Planck equation for the It\^o stochastic differential equation on a manifold was originally studied in the seminal works of~\cite{ito1950stochastic,ito1953stochastic,yosida1949integration}. 
Those formulations, however, were expressed in local coordinates and depended on the specific choice of the Riemannian metric. 
\Cref{thm:FP_Ito} shows that a coordinate-free, intrinsic form of the It\^o Fokker-Planck equation can be derived compactly by introducing the diffusion tensor field. 
This formulation also provides a geometric interpretation of diffusion on manifolds as follows. 
\begin{remark}\label{rem:cont}
    The diffusion term of \eqref{eqn:FP_Ito}, namely $\div_\mu(\div_\mu(p_t D))$, recovers the 
    ``double divergence form'' (\cite{bogachev2022fokker}) that appears in the Fokker--Planck equation on Euclidean space.
    This term represents the \emph{divergence of the diffusive flux} generated by the tensor field $D$ acting on the scalar field $p_t$.
    Geometrically, it measures the total curvature-adjusted spreading of $p_t$ across the manifold according to the anisotropic 
    diffusion geometry defined by $D$; when $D = \sum_{i=1}^n E_i \otimes E_i$, it reduces to the Laplace--Beltrami operator as shown in \Cref{ex:Brownian_Ito}.

    Physically, the inner divergence $\div_\mu(p_t D)$ corresponds to the \emph{diffusive flux}---the net flow of probability induced by random motion---and has the same physical dimension as the \emph{drift flux} $p_t \tilde X$ generated by the vector field.
    The outer divergence then gives the \emph{net rate of accumulation or depletion} at a point due to that flux, describing how $p_t$ redistributes through a medium with spatially varying anisotropic diffusivity---an 
    essential component of the Fokker--Planck equation governing diffusion processes.
    In short, \eqref{eqn:FP_Ito} can be interpreted as a \emph{continuity equation} on a manifold, expressing that the rate of change of the density at a point equals the net inflow of probability at the point, caused by the flux $p_t\tilde X - \tfrac12 \div_\mu(p_t D)$. 
    \Cref{thm:FP_Ito} generalizes this fundamental interpretation of the Fokker-Planck equation to Riemannian manifolds in an intrinsic manner.
\end{remark}



\section{Example}

In this section, we derive the infinitesimal generator and the Fokker-Planck equation for a stochastic process on the two-sphere $\Sph^2=\{x\in\Re^3\,|\, \|x\|x=1\}$, and discuss how it can be utilized for Bayesian estimation. 

\subsection{Two-Sphere}\label{sec:Two-Sphere}

\paragraph{Riemannian Metric}

A point $x\in\Sph^2$ is parameterized by $ x = \begin{bmatrix}\cos\phi \sin\theta, & \sin\phi \sin\theta, & \cos\theta \end{bmatrix},$ where $\theta\in[0,\pi]$ is the co-latitude, and $\phi \in [0,2\pi)$ is the longitude. 

The tangent vectors are given by
\begin{align*}
    \deriv{x}{\theta} & = \begin{bmatrix} \cos\phi \cos\theta & \sin\phi \cos\theta & -\sin\theta\end{bmatrix}^T,\\
    \deriv{x}{\phi} & = \begin{bmatrix} -\sin\phi \sin\theta & \cos\phi \sin\theta & 0\end{bmatrix}^T.
\end{align*}
The resulting components of the metric tensor induced by the standard metric on $\Re^3$ are
\begin{align*}
    \g_{\theta\theta} = 1, \quad
    \g_{\theta\phi}  = 0, \quad
    \g_{\phi\phi} = \sin^2\theta.
\end{align*}
The Christoffel symbols are given by
\begin{align*}
    \Gamma^\theta_{\phi\phi} = - \sin\theta\cos\theta,\quad
    \Gamma^\phi_{\theta\phi} = \Gamma^\phi_{\phi\theta} = \cot\theta,
\end{align*}
and all others vanish. 

\paragraph{Orthonormal Frame}

Using the metric, we can define the following orthonormal frame:
\begin{gather}
    E_\theta  = \deriv{}{\theta},\quad
    E_\phi  = \frac{1}{\sin\theta} \deriv{}{\phi}. \label{eqn:E_Sph}
\end{gather}
It is straightforward to verify $\g(E_i, E_j) = \delta_{ij}$.  

In the subsequent development, we will repeatedly use the covariant derivatives of the orthonormal frame, given by
\begin{gather}
    \nabla_{E_\theta} E_\theta = 0,\quad
    \nabla_{E_\theta} E_\phi = 0,\\
    \nabla_{E_\phi} E_\theta = \cot\theta E_\phi,\quad
    \nabla_{E_\phi} E_\phi = -\cot\theta E_\theta.
\end{gather}

\subsection{Stratonovich SDE}

We choose the drift vector field and the diffusion vector fields using the orthonormal frame as follows. 
The drift vector field is chosen as
\begin{align}
    X(\theta,\phi) = X^\theta(\theta,\phi) E_\theta + X^\phi(\theta,\phi) E_\phi, \label{eqn:X_Sph}
\end{align}
for $X^\theta,X^\phi\in C^\infty(\Sph^2)$. 
And the diffusion vector fields are
\begin{align}
    \sigma_1(\theta,\phi) = \sigma^\theta E_\theta,\quad
    \sigma_2(\theta,\phi) = \sigma^\phi E_\phi, \label{eqn:sigma_Sph}
\end{align}
for constants $\sigma^\theta,\sigma^\phi \in\Re$. 
The resulting SDE is
\begin{align}
    d x(\theta,\phi) = X(\theta,\phi) dt + \sum_{i=1}^2 \sigma_i(\theta,\phi) \circ dW_i. \label{eqn:SDE_Sph}
\end{align}

\paragraph{Infinitesimal Generator}

According to \eqref{eqn:A_Strat}, for $f\in C^\infty(\Sph^2)$, the infinitesimal generator is given by
\begin{align*}
    \mathcal{A} f & = X^\theta E_\theta[f] + X^\phi E_\phi[f]\\
    & \quad + \frac12 \left( (\sigma^\theta)^2 E_\theta[E_\theta[f]]+ (\sigma^\phi)^2 E_\phi[E_\phi[f]] \right).
\end{align*}
Using \eqref{eqn:E_Sph}, this is rearranged into
\begin{empheq}[box=\fbox]{align}
    \mathcal{A} f & = X^\theta \deriv{f}{\theta} + \frac{1}{\sin\theta} X^\phi \deriv{f}{\phi} \nonumber \\
    & \quad + \frac12 \left( (\sigma^\theta)^2 \frac{\partial^2 f}{\partial \theta^2}
    + \frac{(\sigma^\phi)^2}{\sin^2 \theta} \frac{\partial^2 f}{\partial \phi^2} \right). \label{eqn:A_Strat_Sph}
\end{empheq}

\paragraph{Fokker-Planck Equation}

From \eqref{eqn:div}, the divergence of $X$ can be obtained as
\begin{align}
    \div_\mu X & = E^\theta (\nabla_{E_\theta} X) + E^\phi ( \nabla_{E_\phi} X).
\end{align}
Using the Leibniz rule and the covariant derivatives of the frame, we can show
\begin{align*}
    E^\theta(\nabla_{E_\theta} X) & = (\partial_\theta X^\theta) E^\theta(E_\theta) + X^\theta E^\theta(\nabla_{E_\theta} E_\theta) \\
                        & \quad + (\partial_\phi X^\phi) E^\theta(E_\phi) + X^\phi E^\theta(\nabla_{E_\theta} E_\phi)\\
                        & = \deriv{X^\theta}{\theta}.
\end{align*}
Similarly, one can show
\begin{align*}
E^\phi(\nabla_{E_\phi} X) = X^\theta \cot\theta + \frac{1}{\sin\theta} \frac{\partial X^\phi}{\partial \phi},
\end{align*}
so that the divergence on $\Sph^2$ is given by
\begin{align}
    \div_\mu X = \frac{1}{\sin\theta} \frac{\partial(\sin\theta X^\theta)}{\partial \theta}  + \frac{1}{\sin\theta} \frac{\partial X^\phi}{\partial \phi}.\label{eqn:div_Sph}
\end{align}
Using \eqref{eqn:div_Sph} repeatedly, the Fokker-Planck equation in the Stratonovich form \eqref{eqn:FP_Strat}  on the two-sphere is
\begin{empheq}[box=\fbox]{align}
    \frac{\partial p_t}{\partial t} & = -\frac{1}{\sin\theta} \left( \frac{\partial (p_t \sin\theta X^\theta)}{\partial \theta} + \frac{\partial (p_t X^\phi)}{\partial \phi} \right) \nonumber \\
                                    &\quad + \frac{(\sigma^\theta)^2}{2\sin\theta} \frac{\partial^2 (p_t \sin\theta)}{\partial \theta^2} + \frac{(\sigma^\phi)^2}{2\sin^2\theta} \frac{\partial^2 p_t}{\partial \phi^2}. \label{eqn:FP_Strat_Sph}
\end{empheq}


\subsection{It\^{o} SDE}
Next, we consider the It\^{o} formulation. 
Let $\tilde X\in\mathfrak{X}(\Sph^2)$ be a vector field, written as
\begin{align}
    \tilde X(\theta,\phi) & = \tilde X^\theta(\theta,\phi) E_\theta + \tilde X^\phi(\theta,\phi) E_\phi.
\end{align}
And the diffusion fields are given by \eqref{eqn:sigma_Sph}.
The corresponding It\^{o} SDE is
\begin{align}
    d x(\theta,\phi) = \tilde X(\theta,\phi) dt + \sum_{i=1}^2\sigma_i(\theta,\phi) dW_i. \label{eqn:SDE_Sph_Ito}
\end{align}

Using the covariant derivatives of the frame, we can show
\begin{align*}
    \nabla_{\sigma_1}\sigma_1 = 0, \quad
    \nabla_{\sigma_2}\sigma_2 = -(\sigma^\phi)^2 \cot\theta E_\theta.
\end{align*}
According to \Cref{thm:Strat_Ito}, the It\^{o} formulation \eqref{eqn:SDE_Sph_Ito} equivalent to the Stratonovich SDE \eqref{eqn:SDE_Sph} if
\begin{gather}
    \tilde X^\theta = X^\theta - \frac{1}{2}(\sigma^\phi)^2 \cot\theta ,
    \quad \tilde X^\phi = X^\phi. \label{eqn:Strat_Ito_Sph}
\end{gather}

\paragraph{Infinitesimal Generator}
The diffusion tensor is
\begin{align}
    D =  (\sigma^\theta)^2 E_\theta\otimes E_\theta + (\sigma^\phi)^2 E_\phi\otimes E_\phi,
\end{align}
such that $D^{\theta\theta} = (\sigma^\theta)^2$, $D^{\theta\phi}=D^{\phi\theta}=0$, and $D^{\phi\phi} = (\sigma^\phi)^2$ in coordinates.
Next, from \eqref{eqn:Hess}, the first component of the Hessian is
\begin{align*}
    \mathrm{Hess}_f (E_\theta, E_\theta) & = E_\theta[E_\theta[f]] - (\nabla_{E_\theta} E_\theta) [f] = \frac{\partial^2 f}{\partial\theta^2}.
\end{align*}
Similarly, we can show the other components as
\begin{align*}
    \mathrm{Hess}_f (E_\phi, E_\phi) &= \frac{1}{\sin^2\theta} \frac{\partial^2 f}{\partial \phi^2} + \cot\theta \frac{\partial f}{\partial \theta}, \\
    \mathrm{Hess}_f (E_\theta, E_\phi) &= \frac{1}{\sin\theta} \frac{\partial^2 f}{\partial \theta \partial \phi} - \frac{\cot\theta}{\sin\theta} \frac{\partial f}{\partial \phi}, \\
    \mathrm{Hess}_f (E_\phi, E_\theta) &= \mathrm{Hess}_f (E_\theta, E_\phi).
\end{align*}
Substituting these into \eqref{eqn:A_Ito_D}, and using and \eqref{eqn:div_Sph}, the generator for the It\^{o} SDE is given by
\begin{empheq}[box=\fbox]{align*}
    & \mathcal{A} f = \tilde X^\theta \deriv{f}{\theta} + \frac{1}{\sin\theta} \tilde X^\phi \deriv{f}{\phi} \nonumber \\
    & + \frac12 \left( (\sigma^\theta)^2 \frac{\partial^2 f}{\partial\theta^2} +(\sigma^\phi)^2 \left(\frac{1}{\sin^2\theta} \frac{\partial^2 f}{\partial \phi^2} + \cot\theta \frac{\partial f}{\partial \theta}\right)\right).
\end{empheq}
Substituting \eqref{eqn:Strat_Ito_Sph}, we can verify that the above is equivalent to the generator in the Stratonovich form given by \eqref{eqn:A_Strat_Sph}.

\paragraph{Fokker-Planck Equation}

In \eqref{eqn:FP_Ito}, the divergence in the drift contribution and the outer divergence in the diffusion part can be obtained by \eqref{eqn:div_Sph}.
Thus, we compute the inner divergence acting on the $(2,0)$ tensor $p_t D$. 
From \eqref{eqn:div_tensor}, $\div_\mu (p_t D)$ is the $(1,0)$ tensor obtained by
\begin{align}
    (\div_\mu (p_tD))(\alpha) & = (\nabla_{E_\theta} (p_tD) )(\alpha, E^\theta)\nonumber\\
                              & \quad + (\nabla_{E_\phi} (p_tD))(\alpha, E^\phi), \label{eqn:div_pD_Sph}
\end{align}
for $\alpha\in\Omega^1(M)$.
When $\alpha= E^\theta$, the first term on the right-hand side is obtained from \eqref{eqn:covariant_derivative_tensor} as
\begin{align*}
    (\nabla_{E_\theta} (p_t D)) & (E^\theta, E^\theta)  = (E_\theta[p_t]) D^{\theta\theta} 
    + p_t (\nabla_{E_\theta} D)(E^\theta, E^\theta)\\
                                & = (\sigma^\theta)^2 \deriv{p_t}{\theta}  + p_t \left(E_\theta[D^{\theta\theta}] - 2 D(\nabla_{E_\theta} E^\theta, E_\theta)\right).
\end{align*}
But, $E_\theta[D^{\theta\theta}] = 0$, and $\nabla_{E_\theta} E^\theta =0$ from \eqref{eqn:covariant_derivative_tensor}.
Thus, 
\begin{align*}
    (\nabla_{E_\theta} (p_t D)) (E^\theta, E^\theta) = (\sigma^\theta)^2 \deriv{p_t}{\theta}.
\end{align*}
Similarly, we can obtain the remaining terms of \eqref{eqn:div_pD_Sph} when $\alpha = E^\theta, E^\phi$ as
\begin{align*}
    (\nabla_{E_\theta} (p_t D)) (E^\phi, E^\theta) &= 0, \\
    (\nabla_{E_\phi} (p_t D)) (E^\theta, E^\phi) &= p_t \cot\theta ((\sigma^\theta)^2 - (\sigma^\phi)^2), \\
    (\nabla_{E_\phi} (p_t D)) (E^\phi, E^\phi) &= \frac{(\sigma^\phi)^2}{\sin\theta} \frac{\partial p_t}{\partial \phi}.
\end{align*}
Substituting these back to \eqref{eqn:div_pD_Sph}, the inner divergence of \eqref{eqn:FP_Ito} can be written as $\div_\mu(p_t D) = V = V^\theta E_\theta + V^\phi E_\phi \in \mathfrak{X}(M)$, where
\begin{align}
    V^\theta & = (\sigma^\theta)^2 \deriv{p_t}{\theta} + p_t \cot\theta ((\sigma^\theta)^2 - (\sigma^\phi)^2),\\
    V^\phi & = \frac{(\sigma^\phi)^2}{\sin\theta} \frac{\partial p_t}{\partial \phi}.
\end{align}

The remaining drift term and the outer divergence of \eqref{eqn:FP_Ito} can be obtained by repeatedly applying \eqref{eqn:div_Sph}.
The resulting It\^{o} Fokker-Planck equation is given by
\begin{empheq}[box=\fbox]{align}
    \frac{\partial p_t}{\partial t} & = -\frac{1}{\sin\theta} \left( \frac{\partial (p_t \sin\theta \tilde X^\theta)}{\partial \theta} + \frac{\partial (p_t \tilde X^\phi)}{\partial \phi} \right) \nonumber \\
                                    & \quad + \frac{(\sigma^\theta)^2}{2\sin\theta} \left( \frac{\partial(\sin\theta \partial_\theta p_t)}{\partial \theta}  + \frac{\partial (p_t \cos\theta)}{\partial \theta} \right) \nonumber \\
                                    & \quad + \frac{(\sigma^\phi)^2}{2\sin\theta} \left( \frac{1}{\sin\theta}\frac{\partial^2 p_t}{\partial\phi^2}  -  \frac{\partial (p_t \cos\theta)}{\partial \theta} \right). \label{eqn:FP_Ito_Sph}
\end{empheq}
As an ultimate verification, one can show that substituting the Stratonovich-It\^{o} conversion \eqref{eqn:Strat_Ito_Sph} into \eqref{eqn:FP_Ito_Sph}, we obtain the Stratonovich Fokker-Planck equation \eqref{eqn:FP_Strat_Sph}.

\EditTL{
\subsection{Bayesian Estimation}

The Fokker--Planck equation~\eqref{eqn:FP_Strat_Sph} or~\eqref{eqn:FP_Ito_Sph} can be utilized to develop a Bayesian estimator on the two-sphere.  
Suppose that the measurement model is given by
\begin{equation}
    z = Z(x, v),
\end{equation}
where \( Z: \Sph^2 \times \Re^q \rightarrow \Re^m \) and \( v \in \Re^q \) denotes the measurement noise.  
We assume that the measurement likelihood, namely the conditional probability density \( p(z | x) \) of the measurement \( z \) given the state \( x \in \Sph^2 \), is available.  

Let the initial distribution be \( p_0(x) \), and suppose that a measurement becomes available at \( t = t_1 \).  
The prior distribution can be propagated via the Fokker--Planck equation~\eqref{eqn:FP_Strat_Sph} or~\eqref{eqn:FP_Ito_Sph} to obtain the predicted density \( p_1(x) \), which is then updated by the measurement according to Bayes' rule:
\begin{equation}
    p_1(x | z) \propto p(z | x)\, p_1(x),
\end{equation}
yielding the posterior distribution of the state at \( t = t_1 \) conditioned on the measurement.  
This procedure can be applied recursively as new measurements become available, thereby defining a Bayes filter on the two-sphere.
}

\section{Conclusions}

This paper presented a global, coordinate-free formulation of the Fokker–Planck equation on Riemannian manifolds.  
By deriving both the Stratonovich and It\^{o} forms from first principles, the framework unifies stochastic dynamics, geometry, and tensor analysis under a single intrinsic structure.  
The introduction of the diffusion tensor field enabled an elegant and compact representation of the It\^{o} Fokker-Planck equation as a double divergence.  
Together, these developments establish a rigorous geometric foundation for diffusion and probability transport on nonlinear manifolds, offering a pathway for further advances in stochastic modeling, filtering, and control on curved spaces.

\bibliographystyle{ieeetran}
\bibliography{ref}

\appendix

\subsection{Proof of \Cref{lem:IBP}}\label{sec:IBP}

Here we show \eqref{eqn:IBP}, following the identities presented in~\cite{marsden2013introduction}.
The divergence of a vector field is defined through the Lie derivative as $\div_\mu(X)\,\mu = \mathcal{L}_X \mu$.  
From Cartan's formula, $\mathcal{L}_X \mu = d\, i_X \mu + i_X d\mu$, and since $\mu$ is a top form, $d\mu = 0$.  
Thus, $\div_\mu(X)\,\mu = d\, i_X\mu$.  
Integrating this identity over $M$ gives
\begin{align}
    \int_M \div_\mu(X)\, \mu 
    = \int_M d\, i_X\mu 
    =  \int_{\partial M} \iota^*(i_X \mu),\label{eqn:div_thm0}
\end{align}
where the second equality follows from Stokes’ theorem, and $\iota:\partial M \to M$ denotes the inclusion map.
Here, the pullback $\iota^*$ emphasizes that $i_X\mu$ is evaluated as a differential form on the boundary $\partial M$.

Since $fX$ is also a vector field for any $f\in C^\infty(M)$, \eqref{eqn:div_thm0} implies
\begin{align}
    \int_M \div_\mu(f X)\, d\mu = \int_{\partial M} \iota^*(i_{fX} \mu).\label{eqn:div_thm1}
\end{align}
But the product rule for divergence provides
\begin{align}
    \div_\mu(fX) = f\,\div_\mu(X) + \mathcal{L}_X f.\label{eqn:div_prod}
\end{align}
Substituting \eqref{eqn:div_prod} into \eqref{eqn:div_thm1},
\begin{align}
    \int_M f\, \div_\mu(X)\, \mu  
    + \int_M \mathcal{L}_X f\, \mu =  
    \int_{\partial M} \iota^*(i_{fX} \mu), \label{eqn:integ_part}
\end{align}
which gives an intrinsic form of the integration-by-parts identity.

Setting $f = pq$, the boundary term on the right-hand side vanishes.
Also, using the product rule of the Lie derivative given by $\mathcal{L}_X (pq) = p \mathcal{L}_X q + q \mathcal{L}_X p$, this is written as
\begin{align}
    \int_M pq\, \div_\mu(X)\, \mu + \int_M (p\mathcal{L}_X q + q\mathcal{L}_X p)\, \mu = 0 .
\end{align}
Recall that the Lie derivative is denoted by $\mathcal{L}_X q = X[q]$ as given in \eqref{eqn:lie_derivative}.
Thus, this is rearranged into
\begin{align*}
    \pair{p, X[q]} & = \pair{p, \mathcal{L}_X q} = \pair{-\mathcal{L}_X p - p\, \div_\mu(X), q}\\
                   & = \pair{-\div_\mu(pX), q},
\end{align*}
where the last equality is from the product rule of divergence given by \eqref{eqn:div_prod}.
This shows \eqref{eqn:IBP}.

\subsection{Proof of \Cref{thm:generator_Ito}}\label{sec:generator_Ito}

In coordinates, the full contraction of $D\in\mathcal{T}^{(2,0)}(M)$ and $\mathrm{Hess}_f\in\mathcal{T}^{(0,2)}(M)$ is given by
\begin{align*}
    D:\mathrm{Hess}_f = D^{jk} (\mathrm{Hess}_f)_{jk}.
\end{align*}
Let $\{E_j(x)\}_{j=1}^n$ be an orthonormal basis of $T_xM$ at each $x\in M$ with dual $\{E^j(x)\}_{j=1}^n$.
Then, we have
\begin{gather*}
    D(E^j, E^k) = \sum_{i=1}^m (\sigma_i\otimes\sigma_i)(E^j, E^k) = \sum_{i=1}^m E^j(\sigma_i) E^k(\sigma_i),\\
    (\mathrm{Hess}_f)_{jk} = \mathrm{Hess}_f (E_j, E_k). 
\end{gather*}
Substituting these, and using the bilinearity of the Hessian, 
\begin{align*}
    D:\mathrm{Hess}_f & = \sum_{i=1}^m E^j(\sigma_i) E^k(\sigma_i)\mathrm{Hess}_f (E_j, E_k) \\
    & = \sum_{i=1}^m   \mathrm{Hess}_f (E^j(\sigma_i) E_j, E^k(\sigma_i) E_k).
\end{align*}
Since $\sigma_i = E^j(\sigma_i)E_j$, this reduces to the concise relation given by
\begin{align}
    \sum_{i=1}^m \mathrm{Hess}_f(\sigma_i, \sigma_i)
    = D : \mathrm{Hess}_f. \label{eqn:D_Hess}
\end{align}

In \cite[Theorem 2]{lee2025brownian}, it has been shown that the generator of \eqref{eqn:SDE_Ito} is given by
\begin{align*}
    \mathcal{A}f = \tilde X[f] + \frac12 \sum_{i=1}^m \mathrm{Hess}_f(\sigma_i, \sigma_i),
\end{align*}
using \eqref{eqn:A_Strat} and \eqref{eqn:X_tilde}.
Substituting \eqref{eqn:D_Hess} into the above yields \eqref{eqn:A_Ito_D}.

\subsection{Proof of \Cref{lem:IBP_tensor}}\label{sec:IBP_tensor}

We begin by constructing a vector field $V\in\mathfrak{X}(M)$ from the tensor fields $T \in \mathcal{T}^{(r,s)}(M)$ and $S \in \mathcal{T}^{(s+1,r)}(M)$, by taking the contraction.
The $k$-th component of $V$ is
\[
    V^k = T^{I}_J S^{Jk}_{I},
\]
where $I=(i_1,\ldots, i_r)$, $J=(j_1,\ldots, j_s)$, and $k$ is the $(s+1)$-th contravariant index.
Since $T$ and $S$ have compact support, $V$ is also a vector field with compact support. 
We apply the divergence theorem given by \eqref{eqn:div_thm0} to obtain
\begin{align}
    \int_M \div_\mu (V) \, d\mu = \int_M \nabla_k V^k \, d\mu = 0. \label{eqn:div_V}
\end{align}
We now compute the divergence $\nabla_k V^k$ using the product rule as follows. 
\begin{align*}
    \div_\mu (V) & = \nabla_k  (T^{I}_J S^{Jk}_{I}) \\
                 & = (\nabla_k T^I_J) S^{Jk}_{I} + T^I_J (\nabla_k S^{Jk}_{I}).
\end{align*}
We interpret each term of the right-hand side as follows.
The components of the covariant derivative $\nabla T \in \mathcal{T}^{(r,s+1)}$ are $(\nabla T)^I_{Jk} = \nabla_k T^I_J$. Therefore, the first term is
\begin{align*}
    (\nabla_k T^I_J) S^{Jk}_{I} = (\nabla T)^I_{Jk} S^{Jk}_{I} =\nabla T:S.
\end{align*}
Next, the components of $\div_\mu (S)\in\mathcal{T}^{(s,r)}(M)$ are given by $(\div_\mu S)^{J}{I} = \nabla_k S^{Jk}_I$. 
Thus, the second term is
\begin{align*}
    T^I_J (\nabla_k S^{Jk}_{I}) = T^I_J (\div_\mu (S))^I_J = T: \div_\mu (S). 
\end{align*}
Therefore, $\div_\mu (V) = \nabla T: S + T:\div_\mu S$. 
Substituting this back into \eqref{eqn:div_V} and using \eqref{eqn:inner_L2_tensor}, we obtain \eqref{eqn:IBP_tensor}.

\end{document}